\documentclass[a4paper,11pt]{amsart}

\usepackage{amsfonts,amssymb}
\usepackage{color}
\usepackage{textcomp}

\theoremstyle{plain}

\newtheorem*{thmA}{Theorem A}
\newtheorem*{thmB}{Theorem B}
\newtheorem{thm}{Theorem}[section]
\newtheorem{lem}[thm]{Lemma}
\newtheorem{pro}[thm]{Proposition}
\newtheorem{cor}[thm]{Corollary}

\theoremstyle{definition}

\newcommand{\TT}{\mathcal{T}}

\newcommand{\norm}[1]{{\vert #1 \vert}}

\DeclareMathOperator{\Aut}{Aut}

\DeclareMathOperator{\Stab}{Stab}

\begin{document}

\title{On self-similar finite $p$-groups}

\author[A.\ Babai]{Azam Babai}
\address{Department of Mathematics\\ University of Qom\\ Qom\\ Iran}
\email{a\_babai@aut.ac.ir}

\author[Kh.\ Fathalikhani]{Khadijeh Fathalikhani}
\address{Department of Pure Mathematics\\ Faculty of Mathematical Sciences\\
University of Kashan\\ Kashan 87317-51167\\ Iran}
\email{fathalikhani.kh@ut.ac.ir}

\author[G.A.\ Fern\'andez-Alcober]{Gustavo A.\ Fern\'andez-Alcober}
\address{Department of Mathematics\\ University of the Basque Country UPV/EHU\\
48080 Bilbao, Spain}
\email{gustavo.fernandez@ehu.eus}

\author[M.\ Vannacci]{Matteo Vannacci}
\address{Mathematisches Institut der Heinrich-Heine-Universit\"at, Universit\"atsstr.~1, 40225 D\"usseldorf, Germany}
\email{vannacci@uni-duesseldorf.de}

\thanks{The second author is supported by the University of Kashan, under grant no.\  364988/10.
The third author is supported by the Spanish Government, grants
MTM2011-28229-C02-02 and MTM2014-53810-C2-2-P,
and by the Basque Government, grants IT753-13 and IT974-16.}

\keywords{finite $p$-groups, $p$-adic tree, self-similarity,
virtual endomorphisms\vspace{3pt}}


\begin{abstract}
In this paper, we address the following question: when is a finite $p$-group $G$ self-similar, i.e.\ when can $G$ be faithfully represented as a self-similar group of automorphisms of the $p$-adic tree?
We show that, if $G$ is a self-similar finite $p$-group of rank $r$, then its order is bounded by a function of $p$ and $r$.
This applies in particular to finite $p$-groups of a given coclass.
In the particular case of groups of maximal class, that is, of coclass $1$, we can fully answer the question above: a $p$-group of maximal class $G$ is self-similar if and only if it contains an elementary abelian maximal subgroup over which $G$ splits.
Furthermore, in that case the order of $G$ is at most $p^{\hspace{0.5pt}p+1}$, and this bound is sharp.
\end{abstract}

\maketitle

\section{Introduction}

A lot of interest has been raised in the past decades around groups of automorphisms of rooted trees, since they provide many examples and counterexamples for significant problems in group theory.
Prominent among these examples are the groups constructed by Grigorchuk \cite{gri}, and by Gupta and Sidki \cite{gup-sid}, which act on the $p$-adic tree, i.e.\ a regular rooted tree with $p$ children at each vertex, where $p$ is a prime number.
One of the most important properties of the Grigorchuk and Gupta-Sidki groups is the fact that they are \emph{self-similar\/}.
In other words, for every automorphism $g$ in such a group $G$, and for every vertex $u$ of the tree, the section $g_u$ also belongs to $G$.
(See the beginning of Section 2 for the definition of sections and other relevant concepts from the theory of groups of automorphisms of rooted trees.)

Given an abstract group $G$ and a prime $p$, the question arises as to whether or not $G$ can be faithfully represented as a self-similar group of automorphisms of the $p$-adic tree.
If the answer is positive, then we say that $G$ is self-similar for the prime $p$.
Nekrashevych and Sidki \cite{nek-sid} showed that free abelian groups of finite rank are self-similar for the prime $2$, a result which can similarly be established for all primes
\cite[Section 2.9.2]{nek2}.
On the contrary, if a finitely generated nilpotent group $G$ is self-similar for the prime
$p$ and the corresponding action is transitive on the first level of the tree, then $G$ is either free abelian or a finite $p$-group, as proved by Berlatto and Sidki in
\cite[Corollary 2]{ber-sid}.
Further results about self-similar general abelian groups can be found in \cite{bru-sid}; for example, abelian torsion groups of infinite exponent are not self-similar with level-transitive action.

Not much is known about self-similarity of finite $p$-groups.
We can mention the work of \v{S}uni\'c \cite{sun}, who proved that a finite $p$-group is isomorphic to a self-similar group of automorphisms of the $p$-adic tree with abelian first level stabilizer (so in particular with an abelian maximal subgroup) if and only if it is a split extension of an elementary abelian group by a cyclic group of order $p$.
Actually, as we show in Theorem \ref{abelian maximal}, this characterization extends to all finite $p$-groups with an abelian maximal subgroup, without needing to specify any properties of its action on the tree.
On the other hand, since a direct power of a self-similar group is again self-similar
\cite[Proposition 2.9.3]{nek2}, there are self-similar $p$-groups without abelian maximal subgroups.

In this paper we show that self-similarity is a strong condition to impose on a finite $p$-group.
More specifically, in our first main theorem we prove that, for every prime $p$, there are only finitely many self-similar $p$-groups of a given rank.
This should be compared with the fact that every finite $p$-group can be faithfully represented as a group of automorphisms of the $p$-adic tree, by using its action on the coset tree corresponding to a chief series.

\begin{thmA}
Let $G$ be a finite self-similar $p$-group of rank $r$.
Then the order of $G$ is bounded by a function of $p$ and $r$.
\end{thmA}

As a consequence, the order of a self-similar $p$-group can be bounded in terms of the prime $p$ and its coclass.
We can give more detailed information in the case of finite $p$-groups of coclass $1$,
i.e.\ groups of maximal class.
As a matter of fact, in our second main theorem we fully characterize self-similar $p$-groups of maximal class.
Furthermore we obtain the best possible bound for the order of such a group.

\begin{thmB}
Let $G$ be a finite $p$-group of maximal class.
Then $G$ is self-similar if and only if $G$ possesses an elementary abelian maximal subgroup
over which $G$ splits.
If that is the case, then the order of $G$ is at most $p^{\hspace{0.5pt}p+1}$, and this bound is sharp.
\end{thmB}

Thus for groups of maximal class we get the same characterization obtained by
 \v{S}uni\'c, even if most groups of maximal class do not have abelian maximal subgroups.
Observe that Theorem B gives plenty of examples of finite $p$-groups which are not self-similar; in this respect, see Sections 3.1 and 8.2 of the book \cite{lee-mck} by Leedham-Green and McKay.

\vspace{10pt}

\noindent
\textit {Notation.\/}
If $G$ is a finitely generated group, we use $d(G)$ to denote the minimum number of generators of $G$.
On the other hand, if $G$ is a finite $p$-group and $i\ge 0$ is an integer, $\Omega_i(G)$ is the subgroup generated by all elements of $G$ of order at most $p^i$, and $G^{p^i}$ is the subgroup generated by all $p^i$th powers of elements of $G$.

\section{Preliminaries}

We briefly recall some concepts about the theory of groups of automorphisms of regular rooted trees.
Let $X$ be a set of cardinality $k$.
The $k$-adic tree $\TT$  is the graph whose vertices are the elements of the free monoid $X^*$ generated by $X$, and $v$ is a descendant of $u$ if and only if $v=ux$ for some $x\in X$.
The words in $X^*$ of length $n$ form the $n$th level of $\TT$.
An automorphism $g$ of $\TT$ is a bijective map from $X^*$ to itself that preserves incidence.
The section of $g$ at a vertex $u$, which we denote by $g_u$, is the automorphism of
$\TT$ which is defined by the rule
\[
g(uv) = g(u) g_u(v),
\quad
\text{for every $v\in X^*$.}
\]
The automorphisms of $\TT$ form a group $\Aut \TT$ under composition.
A subgroup $G$ of $\Aut \TT$ is \emph{self-similar\/} if $g_u\in G$ for every $u\in X^*$.

A \emph{virtual endomorphism\/} of a group $G$ is a homomorphism
$\phi:H\rightarrow G$, where $H$ is a subgroup of $G$ of finite index.
If $|G:H|=k$ then $\phi$ is said to be a virtual $\frac{1}{k}$-endomorphism.
Also, we say that $\phi$ is \emph{simple\/} if the only subgroup of $H$ which is normal in
$G$ and $\phi$-invariant (in the sense that $\phi(H)\subseteq H$) is the trivial subgroup.
Every self-similar group $G$ of automorphisms of the $k$-adic tree $\TT$ acting transitively on the first level defines a simple virtual $\frac{1}{k}$-endomorphism, by taking $H=\Stab_G(x)$, where $x\in X$, and defining $\phi:H\rightarrow G$ by $\phi(h)=h_x$.
The converse is also true: every simple virtual $\frac{1}{k}$-endomorphism
$\phi:H\rightarrow G$ can be used to define a faithful action of $G$ on $\TT$ that identifies $G$ with a self-similar group acting transitively on the first level, and in which $H$ is the stabilizer of a letter of $X$.
This is accomplished by using a coset tree; see \cite{nek} and \cite{nek2} for further details.

We are interested in the following question: if $p$ is a prime, when can a finite
$p$-group $G$ be faithfully represented as a self-similar group of automorphisms of the
$p$-adic tree?
One can easily see that a non-trivial finite $p$-group with a self-similar action on the
$p$-adic tree is transitive on the first level.
Thus, according to the previous paragraph, the question reduces to this: given a finite $p$-group $G$, when does there exist a simple virtual endomorphism
$\phi:H\rightarrow G$ for some maximal subgroup $H$ of $G$?

As mentioned in the introduction, \v{S}uni\'c gave an answer to this question when $H$ is abelian, which corresponds to the action of $G$ having an abelian first-level stabilizer.
In that case, there exists a simple virtual endomorphism $\phi:H\rightarrow G$ if and only if $H$ is elementary abelian and $G$ splits over $H$.
Next we want to show that \v{S}uni\'c's result applies to all finite $p$-groups with an abelian maximal subgroup, without any requirements about the action of the group on the tree.
We need the following lemma (see \cite[Lemma 4.6 and its proof]{isa}).

\begin{lem}
\label{properties abelian maximal}
Let $G$ be a finite $p$-group with an abelian maximal subgroup $A$.
Then the following hold:
\begin{enumerate}
\item
For every $g\in G\smallsetminus A$, we have $G'=\{[g,a] \mid a\in A\}$.
Thus $|G'|=|A:C_A(g)|$.
\item
If $G$ is not abelian, then $|G:Z(G)|=p|G'|$.
\end{enumerate}
\end{lem}

\begin{thm}
\label{abelian maximal}
Let $G$ be a finite $p$-group with an abelian maximal subgroup.
Then $G$ is self-similar if and only if $G$ possesses an elementary abelian maximal subgroup over which $G$ splits.
\end{thm}

\begin{proof}
By \v{S}uni\'c's result, we only have to prove the `only if' part.
Assume then that $G$ is self-similar, and consider a simple virtual endomorphism
$\phi:H\rightarrow G$, where $H$ is maximal in $G$.
Again by \v{S}uni\'c's theorem, it suffices to prove that $H$ is abelian.
By way of contradiction, suppose that $H$ is not abelian.
Let $A$ be an abelian maximal subgroup of $G$, which exists by hypothesis.

We first prove that $\phi$ is injective.
If there exists $x\in\ker\phi \smallsetminus A$ then, since $A\cap H$ is an abelian maximal subgroup of $H$, we have $H'=[x,A\cap H]$ by
Lemma \ref{properties abelian maximal}.
Thus $\phi(H')=1$ and, since $\phi$ is simple, we have $H'=1$.
This is contrary to our assumption.
Thus $\ker\phi \subseteq A$ and $A$ normalizes $\ker\phi$.
It follows that $\ker\phi\trianglelefteq G$, and consequently $\ker\phi=1$, as desired.

Let $g\in H\smallsetminus A$, and observe that $G'=[g,G]$ by
Lemma \ref{properties abelian maximal}.
If there exists $a\in C_A(g)\smallsetminus H$ then $G=\langle a \rangle H$ and
$G'=[g,H]$.
Hence $G'=H'$.
Since $\phi(H')\subseteq G'$, it follows that $G'=1$, which is a contradiction.
Consequently $C_A(g)\subseteq H$, in other words, $Z(G)\subseteq Z(H)$.
Then, again by Lemma \ref{properties abelian maximal},
\[
|G'| = |A:C_A(g)| = p  |A\cap H:C_A(g)| = p |A\cap H:C_{A\cap H}(g)| = p |H'|. 
\]
Since also $|G:Z(G)|=p|G'|$ and $|H:Z(H)|=p|H'|$, it follows that $Z(G)=Z(H)$.
Now if we put $K=\phi(H)$ then $\phi^{-1}$ defines a simple virtual
$\frac{1}{p}$-endomorphism from $K$ to $G$.
Thus we also have $Z(G)=Z(K)$.
Then
\[
\phi(Z(G)) = \phi(Z(H)) = Z(K) = Z(G),
\]
which is impossible, since $\phi$ is simple.
This final contradiction shows that $H$ must be abelian.
\end{proof}

The last theorem naturally raises the following question: if $G$ is a finite $p$-group with an abelian maximal subgroup $A$, is it true that $G$ is self-similar if and only if $A$ itself is elementary abelian and $G$ splits over $A$?
This is equivalent to asking whether all abelian maximal subgroups of a self-similar
$p$-group are elementary abelian.
As we next see, this question has a positive answer for odd primes, but not for the prime $2$.

\begin{cor}
Let $G$ be a finite $p$-group, where $p$ is an odd prime.
If $G$ has an abelian maximal subgroup $A$, then $G$ is self-similar if and only if $A$ is elementary abelian and $G$ splits over $A$.
\end{cor}

\begin{proof}
By Theorem \ref{abelian maximal}, there exists an elementary abelian maximal subgroup $H$ over which $G$ splits.
Then $A\cap H\subseteq Z(G)$ and $G$ has class at most $2$.
Let $\langle g \rangle$ be a complement to $H$ in $G$, and write $g=ah$ with $a\in A$ and $h\in H$.
Then $a=gh^{-1}$ is of order $p$, since $G$ is a regular $p$-group.
Consequently $A=\langle a, A\cap H \rangle$ is elementary abelian.
\end{proof}

On the other hand, note that the dihedral group $D_8$ is self-similar and possesses a non-elementary abelian maximal subgroup.

\section{Finite $p$-groups of a given rank}

In this section we prove Theorem A, which shows that there are only finitely many self-similar finite $p$-groups of a given rank.
Recall that the (Pr\"ufer) \emph{rank\/} of a finite group is the maximum of the values of $d(H)$ as $H$ ranges over all subgroups of $G$.
The proof of Theorem A relies on the theory of powerful $p$-groups.
A finite $p$-group $G$ is called \emph{powerful\/} if $G'\le G^p$ for $p>2$, or $G'\le G^4$ for $p=2$.
The main results about powerful $p$-groups can be found in the books \cite{DDSMS} and \cite{khu}.

Given a finite group $G$, a tuple $(a_1,\ldots,a_d)$ is said to be a \emph{basis\/} of $G$ if every $g\in G$ can be uniquely written in the form $g=a_1^{n_1}\ldots a_d^{n_d}$ with $0\le n_i < |a_i|$ for $i=1,\ldots,d$.
By \cite[Chapter 2, Exercise 9]{DDSMS}, every powerful $p$-group has bases.
On the other hand, if $G=\langle x_1,\ldots,x_r \rangle$ is a powerful $p$-group then
$G=\langle x_1 \rangle \ldots \langle x_r \rangle$ (see \cite[Corollary 2.8]{DDSMS}).
As a consequence, all bases of $G$ have $d(G)$ elements, and the order of the elements in a basis of $G$ is irrelevant.
In the remainder, when we consider a basis $(a_1,\ldots,a_d)$ of a powerful $p$-group, we will always assume that $|a_1|\ge \cdots \ge |a_d|$.

As we next see, a basis of a powerful $p$-group $G$ can be used to obtain the subgroup $\Omega_1(G)$ if $p$ is odd, and the same result is true for $G^2$ if $p=2$.

\begin{lem}
\label{omega1pow}
Let $G$ be a powerful $p$-group, and let $N=G^2$.
If $\{a_1,\ldots,a_d\}$ is a basis of $N$, and $|a_i|=m_i$ for every $i=1,\ldots,d$, then
\[
\Omega_1(N) = \langle a_i^{m_i/p} \mid i=1,\ldots,d \rangle.
\]
\end{lem}

\begin{proof}
Since $N$ is powerful and $d=d(N)$, we have $|N:N^p|=|N:\Phi(N)|=p^d$.
It follows that $\norm{\Omega_1(N)}=p^d$, by \cite[Theorem 1]{fer2}.
On the other hand, since $\{a_1,\ldots,a_d\}$ is a basis of $N$, the subgroup
\[
\langle a_i^{m_i/p} \mid i=1,\ldots,d \rangle
\]
of $\Omega_1(N)$ has cardinality at least $p^d$.
This completes the proof.
\end{proof}

Of course, $G^2$ coincides with $G$ if $p$ is odd, and stating the last lemma for $G^2$ is only a trick to avoid a case distinction according as $p$ is odd or $p=2$.

\vspace{5pt}

We can now prove Theorem A.
We use the expression `$(p,r)$-bounded' as shorthand for `bounded above by a function of $p$ and $r$'.

\begin{thm}
\label{thm A}
Let $G$ be a self-similar finite $p$-group of rank $r$.
Then the order of $G$ is $(p,r)$-bounded.
\end{thm}

\begin{proof}
Let $\phi:H\rightarrow G$ be a simple virtual endomorphism from a maximal subgroup $H$ of $G$.
By Proposition 2.12 and Theorem 2.13 of \cite{DDSMS}, there exists a normal subgroup $P$ of $G$ of $(p,r)$-bounded index with the property that every subgroup of $P$ which is normal in $G$ is powerful.
If we put $N=(P\cap H)^2$, then $N$ has also $(p,r)$-bounded index in $G$, and it suffices to show that $N$ has $(p,r)$-bounded order.

Let us consider a basis $(a_1,\ldots,a_d)$ of the powerful group $N$, and observe that $d=d(N)\le r$.
Set $\norm{a_i}=m_i$ for $i=1,\ldots,d$, and $m_{d+1}=1$.
If $e$ denotes the exponent of $G/N$, we claim that $m_i\le e m_{i+1}$ for every $i=1,\ldots,d$.
By way of contradiction, assume that $m_j> e m_{j+1}$ for some $j$, and for simplicity, put $m=m_{j+1}$.
Now, we have
\begin{equation}
\label{phi of omega}
\phi(\Omega_1(N^{em})) \subseteq \Omega_1(G^{em}) \subseteq \Omega_1(N^m),
\end{equation}
since $G^e\subseteq N$.
On the other hand, by \cite[Theorem 2.7]{DDSMS},
\[
N^m = \langle a_1^m \rangle \ldots \langle a_d^m \rangle
= \langle a_1^m \rangle \ldots \langle a_j^m \rangle,
\]
and $(a_1^m,\ldots,a_j^m)$ is a basis of $N^m$.
Since $|a_i|>em$ for $i=1,\ldots,j$, we similarly obtain that
$(a_1^{em},\ldots,a_j^{em})$ is a basis of $N^{em}$.
By Lemma \ref{omega1pow},
\[
\Omega_1(N^m) = \langle a_1^{m_1/p},\ldots,a_j^{m_j/p} \rangle
= \Omega_1(N^{em}).
\]
Thus, by (\ref{phi of omega}), we get
$\phi(\Omega_1(N^{em})) \subseteq \Omega_1(N^{em})$.
Since $\phi$ is simple, this implies that $\Omega_1(N^{em})=1$.
This contradiction proves the claim.
It follows that $m_i\le e^{d-i+1}$ for all $i=1,\ldots,d$, and consequently
\[
|N| = m_1\ldots m_d \le e^{d(d+1)/2}.
\]
Since $d\le r$ and $e$ is $(p,r)$-bounded, we conclude that the order of $N$ is $(p,r)$-bounded, as desired.
\end{proof}

\begin{cor}
Let $G$ be a self-similar finite $p$-group of coclass $s$.
Then the order of $G$ is $(p,s)$-bounded.
\end{cor}

\begin{proof}
According to \cite[Theorems 1.2 and 1.5]{sha}, $G$ contains a powerful characteristic subgroup $N$ with $(p,s)$-bounded index and at most $p^{s+1}$ generators.
Thus $\mathrm{rk}(N)\le p^{s+1}$, by \cite[Theorem 2.9]{DDSMS}.
Consequently $G$ has $(p,s)$-bounded rank, and the result follows from
Theorem~\ref{thm A}. 
\end{proof}

\begin{cor}
Let $G$ be a self-similar finite $d$-generator $p$-group of class $c$.
Then the order of $G$ is $(p,d,c)$-bounded.
\end{cor}

\begin{proof}
This follows immediately from Theorem~\ref{thm A}, since the rank of a $d$-generator nilpotent group of class $c$ is $(d,c)$-bounded.
\end{proof}

Of course, the number $d$ of generators is necessary in the bound of the previous corollary, since elementary abelian $p$-groups are all self-similar.

\vspace{8pt}

The following result, whose proof is contained in the proof of Theorem \ref{thm A}, may be of independent interest.

\begin{pro}
Let $G$ be a self-similar finite $p$-group, and let $\phi:H\rightarrow G$ be a simple virtual endomorphism, with $H$ maximal in $G$.
If $U\subseteq H$ is a uniform normal subgroup of $G$ then
$\exp(U)\le \exp(G/U)$ if $p$ is odd, and $\exp(U)\le 4\exp(G/U)$ if $p=2$.
\end{pro}

\begin{proof}
Put $N=U^2$, and use the same notation as in the proof of Theorem \ref{thm A}.
Since $N$ is uniform, we have $\exp(N)=m_d$, and on the other hand,
$m_d\le \exp(G/N)$ by the proof of Theorem \ref{thm A}.
Thus $\exp(N)\le \exp(G/N)$, and the result follows.
\end{proof}

\section{$p$-groups of maximal class}

In this final section, we deal with finite $p$-groups of coclass $1$, also known as $p$-groups of maximal class.
We prove Theorem B, which gives a full characterization of self-similar $p$-groups of maximal class and determines the best bound for the order of a self-similar $p$-group of maximal class.
Observe that, by Theorem \ref{abelian maximal}, we need only consider $p$-groups of maximal class of order at least $p^4$.

Before proceeding, we need to introduce some concepts and results from the theory of groups of maximal class, for which we refer the reader to \cite{fer} and \cite[Chapter III, Section 14]{hup}.
A $p$-group $G$ of order $p^n$ with $n\ge 2$ is said to be of maximal class if its nilpotency class is $n-1$.
If we put $G_i=\gamma_i(G)$ for every $i\ge 2$, this means that $|G:G_2|=p^2$ and
$|G_i:G_{i+1}|=p$ for every $i=2,\ldots,n-1$.
The upper central series of a $p$-group of maximal class coincides with its lower central series; in particular, $Z(G)$ is of order $p$ for $n\ge 3$.
Actually, it is easy to describe all normal subgroups of $G$: we only have to add the $p+1$ maximal subgroups to the terms of the lower central series of $G$.

For a group of maximal class $G$ of order $p^n\ge p^4$, we put $G_1=C_G(G_2/G_4)$, so that $|G:G_1|=|G_1:G_2|=p$.
Then the largest integer $\ell\in\{0,\ldots,n-3\}$ such that $[G_i,G_j]\le G_{i+j+\ell}$ for all
$i,j\ge 1$ is called the \emph{degree of commutativity\/} of $G$, and we write $\ell=\ell(G)$.
A fundamental fact in the theory of $p$-groups of maximal class is that $\ell(G/Z(G))\ge 1$ for all $n\ge 5$.
As a consequence, all the so-called two-step centralizers $C_G(G_i/G_{i+2})$ (which are maximal subgrups of $G$) coincide with $G_1$ for $i=2,\ldots,n-3$, and the remaining two-step centralizer, i.e.\ $C_G(G_{n-2})$, is equal to $G_1$ if and only if $\ell(G)\ge 1$.
Another key result that we will need is that a finite $p$-group is of maximal class if and only if it contains an element $s$ whose centralizer is of order $p^2$.
The elements satisfying this property are exactly the elements in the difference
$G\smallsetminus (G_1\cup C_G(G_{n-2}))$, and they are called
\emph{uniform elements\/} of $G$.
If $s$ is a uniform element, then the order of $s$ is $p$ or $p^2$, and $s^p\in Z(G)$.
Also, for every $1\le i\le n-2$, if $x\in G_i\smallsetminus G_{i+1}$ then we have
$[x,s]\in G_{i+1}\smallsetminus G_{i+2}$.

Our analysis of the representation of a $p$-group of maximal class $G$ as a self-similar group of automorphisms of the $p$-adic tree requires a thorough understanding of the subgroups of $G$ which contain a uniform element.
Certainly, these groups are of maximal class again, but much more can be said.
We need the following straightforward lemma.

\begin{lem}
\label{ll}
Let $G$ be a $p$-group of maximal class of order $p^n\ge p^4$, and let $K$ be a subgroup of $G$ which contains a uniform element.
If there exists an element $x$ of $K$ lying in the difference $G_t \setminus G_{t+1}$, then
$G_t \leq K$.
\end{lem}

\begin{proof}
If $s\in K$ is a uniform element of $G$, then $[x,s]\in G_{t+1} \setminus G_{t+2}$,
$[x,s,s]\in G_{t+2} \setminus G_{t+3}$, and so on until we get an element in
$G_{n-1}\smallsetminus G_n$.
Since all these elements lie in $K$, we conclude that $G_t\le K$.
\end{proof}

\begin{pro}
\label{gen}
Let $G$ be a $p$-group of maximal class of order $p^n\geq p^4$, and let $K$ be a subgroup of $G$ which contains a uniform element $s$.
If the order of $K$ is $p^{n-t}$, then $K=\langle s, G_{t+1}\rangle$.
\end{pro}

\begin{proof}
Since $G=\langle s\rangle G_1$ and $s\in K$, we have $K=\langle s\rangle(K\cap G_1)$, by Dedekind's law.
Hence $K\cap G_1$ is of order $p^{n-t-1}$.
The result is obvious for $t=n-1$, so we assume that $t\le n-2$.
Then $K\cap G_1\ne 1$, and so we can consider
\[
j = \min \{ i \mid 1\le i\le n-1,\ K\cap (G_i\smallsetminus G_{i+1})\ne \varnothing \}.
\]
Now it is clear from Lemma~\ref{ll} that $K\cap G_1=G_j$.
Since $|G_j|=p^{n-j}$, we conclude that $j=t+1$.
Thus $K=\langle s,G_{t+1} \rangle$, as desired.
\end{proof}

\begin{pro}
\label{max}
Let $G$ be a $p$-group of maximal class of order at least $p^4$, and let $H$ be a maximal subgroup of $G$ different from $G_1$.
Then, there is no simple virtual endomorphism from $H$ to $G$.
\end{pro}

\begin{proof}
First of all, we observe that $H'=[H,G_2]$, since $H/G_2$ is cyclic.
The commutator $[H,G_2]$ is a normal subgroup of $G$ which is strictly contained in $G_2$, so it must coincide with $G_i$ for some
$i\ge 3$.
Now $[H,G_2]$ is not contained in $G_4$, since otherwise $H=G_1$.
Hence $[H,G_2]=G_3$.

By way of contradiction, let us assume that $\phi:H\rightarrow G$ is a simple virtual endomorphism.
Then $\phi(H)$ is a proper subgroup of $G$, and we can consider a maximal subgroup
$M$ of $G$ containing $\phi(H)$.
Now we have
$$
\phi(H')=\phi(H)'\le M' = [M,G_2]\leq G_3=H',
$$
which is impossible, since $\phi$ is simple and $H'\ne 1$.
\end{proof}

At this moment, we can prove the second part of Theorem B, namely the bound for the order of a self-similar $p$-group of maximal class.
This bound is sharp, since the wreath product $C_p\wr C_p$ is a self-similar $p$-group of maximal class of order $p^{\hspace{0.5pt}p+1}$.

\begin{cor}
Let $G$ be a $p$-group of maximal class.
If $G$ is self-similar then $|G|\le p^{\hspace{0.5pt}p+1}$.
\end{cor}

\begin{proof}
Assume that $|G|>p^{\hspace{0.5pt}p+1}$.
By Theorem 4.9 of \cite{fer}, we have $G_1^p=G_p$.
Since $s^p\in Z(G)$ for every $s\in G\smallsetminus G_1$, it follows that
$G^p=G_1^p\ne 1$.
Now let $\phi:H\rightarrow G$ be a simple $\frac{1}{p}$-endomorphism of $G$, and observe that $H=G_1$ by Proposition \ref{max}.
Then $\phi(G^p)=\phi(G_1^p)\subseteq G^p$, which is a contradiction, since $\phi$ is simple.
\end{proof}

The following result completes the proof of Theorem B.

\begin{thm}
\label{main}
Let $G$ be a $p$-group of maximal class of order at least $p^4$.
Then $G$ is self-similar if and only if $G_1$ is elementary abelian, and $G$ is a split extension over $G_1$.
\end{thm}

\begin{proof}
By \v{S}uni\'c's result, we only need to prove the `only if' part of the statement of the theorem.
So assume that $G$ is self-similar, and let $\phi:H\rightarrow G$ be a simple virtual endomorphism, where $H$ is a maximal subgroup of $G$.
By Proposition \ref{max}, $H$ must be $G_1$.

Let $G$ be of order $p^n$, and suppose first that $\ell(G)\ge 1$.
Since $\phi$ is simple, it follows that $\phi(G_1)\nsubseteq G_1$.
If $s\in \phi(G_1)\smallsetminus G_1$, then $s$ is a uniform element of $G$, since
$C_G(G_{n-2})=G_1$ in this case.
Consequently, by Proposition~\ref{gen}, we have $\phi(G_1)=\langle s, G_{t+1}\rangle$, for some integer $t$.
If $|\phi(G_1)|\geq p^3$, then the centre of $\phi(G_1)$ is of order $p$, since this group is also of maximal class.
As a consequence, we get $Z(\phi(G_1))=Z(G)$.
On the other hand, since $Z(G_1)$ is a normal subgroup of $G$, it follows that $Z(G)\leq Z(G_1)$.
Consequently,
$$
\phi(Z(G_1))\leq Z(\phi(G_1))=Z(G)\leq Z(G_1),
$$
which is a contradiction, since $\phi$ is simple.
Thus $|\phi(G_1)|=p$ or $p^2$, and then $\phi(G_1')=\phi(G_1)'=1$.
Since $\phi$ is simple, this implies that $G_1$ is abelian.
By \v{S}uni\'c's result, we know that $G_1$ is elementary abelian and $G$ splits over $G_1$, as desired.

Assume now that $\ell(G)=0$.
Since $Z(G_1)$ is normal in $G$, we must have $Z(G_1)=G_i$ for some $i$.
But we have $[G_1,G_{n-2}]\ne 1$, since $G_1\ne C_G(G_{n-2})$ in this case.
Thus $Z(G_1)=G_{n-1}$ is of order $p$.
If $\phi$ is not injective, then $Z(G_1)\leq \ker \phi$.
Therefore, $\phi(Z(G_1))=1$, which is a contradiction, since $\phi$ is simple.
On the other hand, if $\phi$ is injective, then $\phi(G_1)$ is a maximal subgroup of $G$. Since $G_1'=[G_1,G_2]\le G_4$, we have $|G_1:G_1'|\ge p^3$, and $G_1$ is not of maximal class.
So neither $\phi(G_1)$ is of maximal class.
But we know that all maximal subgroups of $G$ other than $G_1$ and $C_G(G_{n-2})$ are of maximal class, and also $\phi(G_1)\ne G_1$, since $\phi$ is simple.
Thus we necessarily have $\phi(G_1)=C_G(G_{n-2})$.
This implies that $G_1$ and $C_G(G_{n-2})$ are isomorphic, but notice that $|Z(G_1)|=p$, while $Z(C_G(G_{n-2}))$ contains $G_{n-2}$, which is of order $p^2$.
This contradiction shows that we cannot have $\ell(G)=0$, which completes the proof of the theorem.
\end{proof}

If $G$ is a $p$-group of maximal class with $G_1$ elementary abelian, one cannot guarantee that $G$ is a self-similar $p$-group.
Indeed, it may perfectly happen that all elements in $G\smallsetminus G_1$ (i.e.\ all uniform elements) are of order $p^2$.
However, we have the following result.

\begin{cor}
Let $G$ be a $p$-group of maximal class of order at most $p^{\hspace{0.5pt}p+1}$ in which $G_1$ is abelian.
Then $G/Z(G)$ is a self-similar $p$-group.
\end{cor}

\begin{proof}
This is an immediate consequence of Theorem \ref{main}, since by
Theorem 4.7 of \cite{fer} the quotient group $G/Z(G)$ is of exponent $p$.
\end{proof}

\noindent
\textit{Acknowledgment\/}.
We thank Prof.\ Said Sidki for helpful comments.

\end{document}